%% file: ex_article.tex
\documentclass[onefignum,onetabnum]{siamart171218}

\input{ex_shared}

\ifpdf
\hypersetup{
  pdftitle={Euclidean Distance Bounds for LMI Analytic Centers},
  pdfauthor={B. Roig-Solvas and M. Sznaier}
}
\fi

\myexternaldocument{ex_supplement}

\begin{document}

\maketitle

% REQUIRED
\begin{abstract}
Linear matrix inequalities (LMIs) are ubiquitous in modern control theory, as well as in a variety of other fields in science and engineering. Their analytic centers, i.e. the maximum determinant elements of the feasible set spanned by these LMIs, are the solution of many well-known problems in statistics, communications, geometry, and control and can be approximated to arbitrary precision by semidefinite programs (SDPs). 
The quality of these approximations is measured with respect to the difference in log-determinant of both the exact and the approximate solution to these SDPs, a quantity that follows directly from the duality theory of semidefinite programming. 
However, in many applications the relevant parameters  are functions of the entries of the LMI argument $X$. In these cases it is of interest to develop metrics that quantify the quality of approximate solutions based on the error on these parameters, something that the log-determinant error fails to capture due to the non-linear interaction of all the matrix entries. 
In this work we develop upper bounds on the Frobenius norm error between suboptimal solutions $X_f$ and the exact optimizer $X_*$ of maximum determinant problems, a metric that provides a direct translation to the entry-wise error of $X$ and thus to the relevant parameters of the application.
We show that these bounds can be expressed through the use of the Lambert function $W(x)$, i.e. the solution of the equation $W(x) e^{W(x)}= x$, and derive novel bounds for one of its branches to generate efficient closed-form bounds on the Euclidean distance to the LMI analytic center. Finally, we test the quality of these bounds numerically in the context of interior point methods termination criteria.
\end{abstract}

% REQUIRED

% REQUIRED
\input{Tex/01-.Intro}
\input{Tex/02-.Section2}
\input{Tex/03-.Section3}
\input{Tex/031-Section31}
\input{Tex/04-.Section4}

\input{Tex/05-.Section5}

\section*{Acknowledgments}
The authors are indebted to the anonymous reviewers for the many suggestions to improve the proofs in the original manuscript. This work was partially supported by NSF grants  IIS--1814631, ECCS—1808381 and CNS--2038493, and AFOSR grant FA9550-19-1-0005.

\bibliographystyle{siamplain}
\bibliography{references}
\end{document}

%% file: ex_shared.tex
\usepackage{lipsum}
\usepackage{amsfonts}
\usepackage{graphicx}
\usepackage{epstopdf}
\usepackage{algorithmic}
\usepackage{cite}
\usepackage{amsmath,amssymb}
\usepackage{textcomp}
\usepackage{xcolor}
\usepackage{bbm}
\usepackage{bm}
\ifpdf
  \DeclareGraphicsExtensions{.eps,.pdf,.png,.jpg}
\else
  \DeclareGraphicsExtensions{.eps}
\fi

\newsiamremark{remark}{Remark}
\newsiamremark{hypothesis}{Hypothesis}
\crefname{hypothesis}{Hypothesis}{Hypotheses}
\newsiamthm{claim}{Claim}

\headers{Euclidean Distance Bounds for LMI Analytic Centers}{B. Roig-Solvas and M. Sznaier}

\title{Euclidean Distance Bounds for LMI Analytic Centers using a Novel Bound on the Lambert Function}

\author{Biel Roig-Solvas\thanks{Department of Electrical and Computer Engineering, Northeastern University, Boston, USA
  (\email{biel@ece.neu.edu}, \email{msznaier@coe.neu.edu}).}
  \and Mario Sznaier\footnotemark[1]}

\usepackage{amsopn}

\makeatletter
\newcommand*{\addFileDependency}[1]{% argument=file name and extension
  \typeout{(#1)}% latexmk will find this if $recorder=0 (however, in that case, it will ignore #1 if it is a .aux or .pdf file etc and it exists! if it doesn't exist, it will appear in the list of dependents regardless)
  \@addtofilelist{#1}% if you want it to appear in \listfiles, not really necessary and latexmk doesn't use this
  \IfFileExists{#1}{}{\typeout{No file #1.}}% latexmk will find this message if #1 doesn't exist (yet)
}
\makeatother

\newcommand*{\myexternaldocument}[1]{%
    \externaldocument{#1}%
    \addFileDependency{#1.tex}%
    \addFileDependency{#1.aux}%
}

%% file: Tex/01-.Intro.tex
\section{Introduction}\label{1:sec}
Since their popularization in the 1990s, linear matrix inequalties (LMIs) have become an essential tool for the study of many problems in control theory  (e.g $\mathcal{H}_2$ and $\mathcal{H}_\infty$ synthesis, robust stability analysis). Further, their wide applicability to problems involving matrix spectra together with the development of efficient numerical methods to solve convex optimization problems \cite{Karmakar1984,Nesterov1994} have made LMIs and semidefinite programming (SDP) ubiquitous in science and engineering \cite{Boyd1994,Vander1996,Wolkowicz2000,Anjos2011}. A particular problem involving LMIs is that of computing the analytic center of a set of LMIs, i.e. the maximum determinant matrix inside the feasible set spanned by those LMIs. This problem, also known as the maximum determinant (\textit{maxdet}) problem, has applications in a variety of fields like statistics, geometry, communications and control (see \cite{Vander1998} for a thorough review of the applications of \textit{maxdet}). The \textit{maxdet} problem can be stated as:
\begin{equation}
    \operatorname*{max}_X\; |X| \quad \text{s.t.}\;X \in \mathcal{D}, X \succeq 0
    \label{1:eq:maxdet}
\end{equation}
where $\mathcal{D}$ is a convex set and $X \succeq 0$ indicates that $X$ is positive semidefinite, and to avoid the degenerate case we assume that there exists one $X\in \mathcal{D}$ for which $X\succ 0$. This problem is often cast as the equivalent problem:
\begin{equation}
    \operatorname*{min}_X\; -\log\left(|X|\right) \quad \text{s.t.}\;X \in \mathcal{D}, X \succeq 0
    \label{1:eq:minlogdet}
\end{equation}
that is more tractable due to the convexity, smoothness and self-concordance of the $-\log\left(|X|\right)$ function in the cone of PSD matrices \cite{Boyd2004}. Problems of the form \eqref{1:eq:minlogdet} can be numerically solved using SDP algorithms such as interior point methods, which can solve \eqref{1:eq:minlogdet} up to any accuracy $\epsilon$, as measured by the optimality gap:
\begin{equation}
    \epsilon \geq \log\left(|X_*|\right) -\log\left(|X_f|\right) \geq 0
    \label{1:eq:costbound}
\end{equation}
where $X_*$ is the exact optimizer of \eqref{1:eq:minlogdet} and $X_f$ is a feasible solution to the same problem. Bounds of the form \eqref{1:eq:costbound} on the suboptimality of the cost function are usually easy to obtain, through self-concordance of the cost function or using duality \cite{Boyd2004} (in fact, standard interior point solvers routinely compute this bound to decide their termination).

\textcolor{black}{However, in a large number of  applications the parameters of interest  either depend on entries of the matrix variable $X$ or are entries of $X$ themselves. For instance, set-theoretic  estimation and control methods seek to find sets that satisfy suitable properties \cite{blanchini2007}. While usually convex, these sets are generally hard to characterize and thus, in many instances are approximated either by ellipsoids of the form $x^T X x \leq 1$ or the intersection of such ellipsoids. In the context of worst case estimation, the goal here is to use noisy measurement to find the minimum volume ellipsoid known to contain the state of a system \cite{Schweppe68}, which can be posed as a \textit{maxdet} problem of the form \eqref{1:eq:minlogdet} \cite{Vander1996}. In the context of robust control analysis and synthesis, the goal is to estimate the largest ellipsoid that can be steered to the origin by bounded controls, in the presence of bounded uncertainty and disturbances
\cite{Boyd1994,HENRION1999}. A related problem, arising in the context of peak minimization,  is to use bounded controls to minimize the size of an ellipsoid known to contain the set of all reachable states from the origin, when the system is driven by a bounded-energy perturbation (see chapter 6 in \cite{Boyd1994}). In these cases, the quantities  of interest --bounds on the worst case estimation error, whether or not given points are inside the region of attraction of the origin or whether the reachable set  intersects an ``unsafe" region-- are related to the geometry of these ellipsoids, rather than merely their volume.  Thus, bounds on the log-determinant of the form \eqref{1:eq:costbound} are poor metrics, since the volume of the ellipsoid corresponding to a feasible solution $X_f$ could be close to the optimal volume, while $X_f$ itself is not close to $X_*$.
 On the other hand,  the metric $||X_*-X_f||_F
^2$, captures the difference between the geometry of the ellipsoids. Hence, for these applications, it  provides a more accurate termination criteria than the usual optimality gap.}
\\
In this work we combine the easily accessible bound $\epsilon$ from \eqref{1:eq:costbound} with the interpretability of the Frobenius norm bounds to provide a bound of the form:
\begin{equation}
        ||X_f-X_*||_F^2 \leq  ||X_f||_2^2 \left(\frac{\sqrt[3]{1-e^{-\epsilon}}}{1-\sqrt[3]{1-e^{-\epsilon}}}\right)^2
        \label{1:eq:finalbound}
\end{equation}
Asymptotically, this bound behaves like $||X_f||_2^2\,\epsilon^{2/3}$ as $\epsilon \to 0$, tightening the bound on the Frobenius norm as the optimality gap \eqref{1:eq:costbound} decreases. The bound is developed by analyzing the properties of the optimizers of \eqref{1:eq:minlogdet} together with the derivation of a novel bound on the principal branch of the Lambert function $W_0(x)$, i.e. the solution of the equation $W_0(x)\,e^{W_0(x)} = x$, for which no closed-form is available. In Section 2 we analyze problem \eqref{1:eq:minlogdet} and derive bounds on the Frobenius norm $ ||X_f-X_*||_F^2$ as a function of the log-determinant bound $\epsilon$ and the Lambert function $W(x)$. In Section 3 we present the novel bound on $W_0(x)$ and in Section 4 we provide the proof for the proposed bound \eqref{1:eq:finalbound}. In Section 5 we test the quality of the bound in a numerical setting and present conclusions and future work in Section 6. 

%% file: Tex/02-.Section2.tex
\section{Bound Derivation}
In this section we derive the bounds on the Frobenius norm presented in \eqref{1:eq:finalbound}. A general form of such bounds is given by:
\begin{equation}
    ||X_*-X_f||_F^2 \leq  ||X_f||_2^2 \,g(\epsilon)
    \label{2:eq:genform}
\end{equation}
We start the derivation by defining the matrix $R = X_*^{-1/2} X_f^{1/2}$. We note that $X_*$ is positive-definite and thus invertible by assumption due to the definition of the set $\mathcal{D}$. Furthermore, assuming finitiness of $\epsilon$ we get that $X_f$ is also strictly positive-definite and thus invertible too. This definition leads to:
\begin{equation}
    \left(X_f-X_*\right) = X_f^{1/2} X_f^{-1/2}\left(X_f-X_*\right)X_f^{-1/2} X_f^{1/2} = X_f^{1/2} \left(I-R^{-1} R^{-T}\right) X_f^{1/2}
\end{equation}
We can use the inequality $||AB||_F^2 \leq ||A||_2^2 \,||B||_F^2$ for symmetric matrices to  bound the Frobenius norm by: 
\begin{equation}
\begin{split}
    ||X_f-X_*||_F^2 =||X_f^{1/2} \left(I-R^{-1} R^{-T}\right) X_f^{1/2}||_F^2 &\leq ||X_f^{1/2}||_2^2 \,|| \left(I-R^{-1} R^{-T}\right) X_f^{1/2}||_F^2\\
    &\leq ||X_f^{1/2}||_2^4\,|| I-R^{-1} R^{-T}||_F^2\\
    &=||X_f||_2^2\,\|I-R^{-1} R^{-T} ||_F^2
\end{split}
\label{2:eq:norm_ineq}
\end{equation}
Combining \eqref{2:eq:genform} and \eqref{2:eq:norm_ineq}, we define $g(\epsilon)$ as the largest squared Frobenius norm $||I-R^{-1} R^{-T} ||_F^2$ that a matrix $R$ can have, given a set of constraints on $R$ which we describe below. For simplicity, we will define these constraints on the symmetric matrix $Q$ instead of $R$, a matrix we define as $Q =  X_*^{-1/2}\left(X_f-X_*\right)X_*^{-1/2}  $ $= R R^T -I$. Denoting by $\sigma_i$ the singular values of $R$, the eigenvalues $\bm{\lambda}$ of $Q$ are given by $\lambda_i = \sigma_i^2 -1$. The eigenvalues of $\left(I-R^{-1} R^{-T}\right)$, whose Frobenius norm we aim to bound, can also be written in terms of $\bm{\lambda}$, as they take the form  $1 - \frac{1}{\sigma_i^2} = \frac{\lambda_i}{1+\lambda_i}$. Next we define the characteristics of the optimization problem that allows us to derive the structure of $g(\epsilon)$ as a function of the eigenvalues $\lambda_i$. 
\subsection{Upper bound $g(\epsilon)$ as an optimization problem }
The structure of the \textit{maxdet} problem provides an additional constraint on the trace of the matrix $Q$, described in the following lemma:
\begin{lemma}
The matrix $Q$ satisfies $\operatorname{Tr}\left(Q\right) \leq  0$.
\end{lemma}
\begin{proof}
Consider the function $F(t) = -\log\left(|X_*- t\left( X_*-X_f\right)|\right)$ and its minimization on the interval $0\leq t\leq 1$. This problem is equivalent to \eqref{1:eq:minlogdet} restricted to the line segment between $X_*$ and $X_f$ and so the optimum must lie at $t=0$, from which follows that its gradient at $t=0$ must be non-negative, i.e. $F'(0) \geq 0$. Expressing $F(t)$ with respect to $Q$ we have that:
\begin{equation}
    F(t) =  -\log\left(|I+ tQ|\right) - \log\left(|X_*|\right) \quad \implies\quad F'(t) = -\text{Tr}((I+tQ)^{-1}Q)
\end{equation}
 Evaluating the negative gradient at $t=0$ leads to $-F'(0) = \operatorname{Tr}\left(Q\right) \leq0$ which completes the proof.  
\end{proof}
Noting that the trace of a symmetric matrix is equivalent to the sum of its eigenvalues, the non-positivity of the trace of $Q$ allows us to derive a bound for the eigenvalues $\lambda_i$ of the form:
\begin{equation}
    \begin{split}
        \epsilon \geq \log\left(|X_*|\right) -\log\left(|X_f|\right) = -\log\left(|I+Q|\right)
        &\geq \operatorname{Tr}\left(Q\right) -\log\left(|I+Q|\right)\\ &= \sum_i^N \lambda_i - \log(1+\lambda_i)
    \end{split}
\end{equation}
We can pose the bound $g(\epsilon)$ on the Frobenius norm $|| I-R^{-1} R^{-T} ||_F^2 = \sum_i \left(\frac{\lambda_i}{1+\lambda_i}\right)^2$ as an optimization problem, where we aim to find the maximum Frobenius norm $|| I-R^{-1} R^{-T} ||_F^2$ given the eigenvalue constraints presented above:
\begin{equation}
\begin{split}
       g\left(\epsilon\right) = \operatorname*{max}_{\lambda_i}\; \sum_i \left(\frac{\lambda_i}{1+\lambda_i}\right)^2 \; \text{s.t.}\;\sum_i^N \lambda_i - \log(1+\lambda_i)\leq \epsilon,\;
       \sum_i^N \lambda_i \leq 0,\; \lambda_i > -1\;\forall i
\end{split}
    \label{2:eq:g_eps}
\end{equation}
We can decouple the variables $\lambda_i$ on the logarithmic constraint by including auxiliary variables $\tau_i$:
\begin{equation}
    \sum_i^N \lambda_i - \log(1+\lambda_i)\leq \epsilon \quad \iff \quad \lambda_i - \log(1+\lambda_i) = \tau_i,\quad \sum_i^N \tau_i\leq \epsilon
\end{equation}
The variables $\lambda_i$ can be isolated in the following way:
\begin{equation}
    \begin{split}
        \lambda_i - \log(1+\lambda_i)&=\tau_i\\
        -1 -\lambda_i + \log(1+\lambda_i)&=-1-\tau_i\\
        e^{-1-\lambda_i } (-1-\lambda_i) &=-e^{-\left(1+\tau_i\right)} \\
         \lambda_i &= - W(-e^{-\left(1+\tau_i\right)}) -1
    \end{split}
    \label{2:eq:lambdatau}
\end{equation}
where in the last step we have used the the product logarithm function $W(x)$ satisfying $W(x)e^{W(x)}=x$. In the next section we explore the properties of this function and present two theorems that will allow us to characterize the solution of the optimization problem \eqref{2:eq:g_eps}.

%% file: Tex/03-.Section3.tex
\section{Bounds on the Lambert function}\label{3:sec}
The Lambert function $W(x)$, also known as the product logarithm function, is the function defined by:
\begin{equation}
    W(x) e^{W(x)} = x
    \label{3:eq:Lambert}
\end{equation}
and when defined over the reals, its domain is $x\geq-e^{-1}$ \cite{Corless1996}. In the range $-e^{-1}\leq x < 0$, $W(x)$ can take two different values, defined by the two branches of the function $W_0(x)$ and $W_{-1}(x)$, shown in Figure \ref{fig:Fig1}.a. The former, also known as the principal branch, satisfies $W_0(x)\geq-1$, while the latter satisfies $W_{-1}(x)\leq -1$, meeting at the branching point $x= -e^{-1}$ where $W_{0}(-e^{-1}) = W_{-1}(-e^{-1}) = -1$. The range of $W(x)$ for $x\geq 0$ belongs to the principal branch.\\
\\
The Lambert function has wide application in many science and engineering disciplines, including physics and astrophysics \cite{Packel2004,Stewart2009,Cranmer2004,Warburton2004}, control theory for the stability analysis of time delayed systems \cite{Yi2010,Yi2010a,Yi2008,Yi2006,Chen2002,Shinozaki2006}, electronics \cite{Banwell2000} and biochemistry \cite{Golicnik2012}, to name a few. While there is no analytic closed-form for $W(x)$, accurate approximations have been developed, some of which presenting remarkable low relative error ($\leq 10^{-16}$)\cite{Boyd1998,Barry1995,Barry2000} through the used of complex formulas and/or iterative schemes. Simpler and more interpretable bounds were later developed for $W_{-1}$ in \cite{Chatz2013,Alzahrani2018}, and for the positive portion of $W_0$ \cite{Hoorfar2008}. Next we develop similar bounds for the negative domain of $W_0$, as seen in Figure \ref{fig:Fig1}.b, a region of the Lambert function for which, to the best of our knowledge, no simple bounds have been previously proposed.\\
\begin{figure}
    \centering
    \includegraphics[width=8cm]{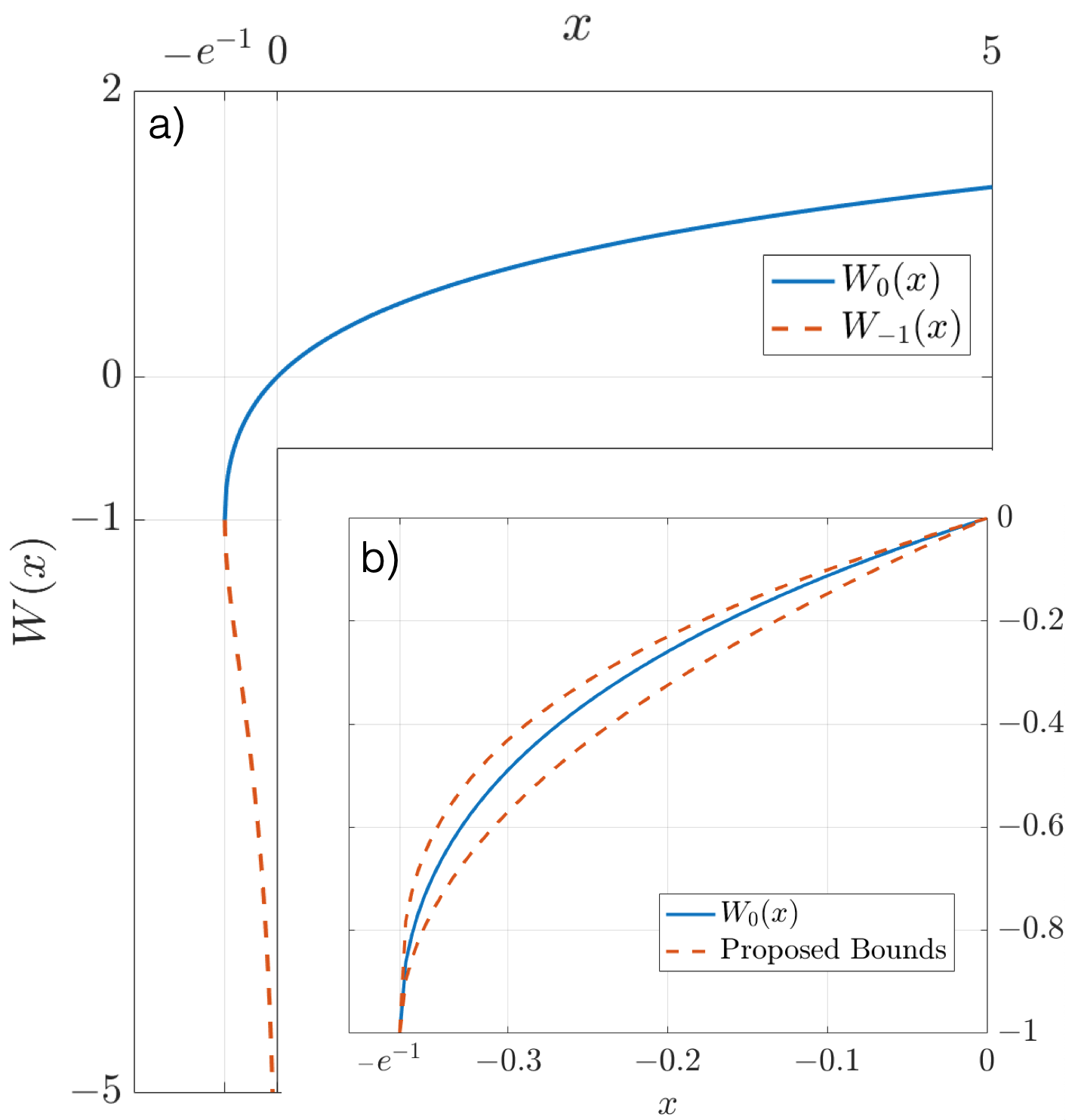}
    \caption{a) The two branches of the Lambert function over the real values: $W_0(x)$ in blue solid line and $W_{-1}(x)$ in red dashed line. b) Detail on the $W_0(x)$ branch over the negative range, in blue solid line, and the upper and lower bounds proposed in this work, in red dashed line.}
    \label{fig:Fig1}
\end{figure}
\\
We begin the derivation by presenting the following bounds.
\begin{lemma}
The function $x-\log(1+x)$ is bounded as follows for $-1<x\leq 0$:
\begin{equation}
    -\log\left(1+x^3\right) \leq x - \log\left(1+x\right) \leq -\log\left( 1-x^2\right)
    \label{3:eq:logbound}
\end{equation}
\end{lemma}
\begin{proof}
For the left hand side, we introduce the function $l(x) =   x - \log\left(1+x\right)+\log\left(1+x^3\right)$. Non-negativity of $l(x)$ in the range $-1<x\leq 0$ is equivalent to satisfaction of  the lower bound in \eqref{3:eq:logbound}. The derivative of $l(x)$ is given by:
\begin{equation}
     l^\prime (x) = 1 - \frac{1}{1+x}+\frac{3x^2}{1+x^3} = \frac{x\left(1+x\right)^3}{(1+x)(1+x^3)}
\end{equation}
which is negative for $-1<x<0$. As $l(0) = 0$, $l(x)$ is non-negative in the proposed range, which completes the proof. Similarly, for the right hand side bound of \eqref{3:eq:logbound}, we have that if  $u(x) =  x - \log\left(1+x\right)+\log\left(1-x^2\right)$ is non-positive in $-1<x<0$ then the bound holds. Its derivative is given by:
\begin{equation}
    u^\prime (x) = 1 - \frac{1}{1+x}-\frac{2x}{1-x^2} = \frac{-x}{(1-x)}   
\end{equation}
which is positive in $-1<x<0$ and thus $u(x)$ is increasing in that range. Combining this with $u(0) = 0$ certifies that $u(x)$ is non-positive in the desired range and finishes the proof.
\end{proof}
We are now ready to present the proposed bound on $W_0$:
\begin{theorem}
In the range $-e^{-1}\leq x < 0$ the function $W_0(x)$ is bounded as follows:
\begin{equation}
    \sqrt{e\,x+1}-1 \leq W_0(x) \leq  \sqrt[3]{e\,x+1}-1
    \label{3:eq:boundlamb}
\end{equation}
\label{3:thm:bounds}
\end{theorem}
\begin{proof}
Applying the change of variables $W_0(x) = -1-\tau$ to equation \eqref{3:eq:Lambert} and inverting signs we get:
\begin{equation}
    \left(1+\tau\right) e^{-1-\tau} = -x
\end{equation}
Applying the logarithm to both sides and changing signs again results in
\begin{equation}
    \tau - \log\left(1+\tau\right) = -\log(-x) -1
\end{equation}
As $W_0(x)$ satisfies $-1\leq W_0(x)\leq 0$ in the interval $-e^{-1}\leq x < 0$, we have that $-1<\tau \leq0$ and thus we can use the bounds from \eqref{3:eq:logbound} to get:
\begin{equation}
\begin{split}
    -\log\left(1+\tau^3\right) &\leq -\log(-x) -1 \leq -\log\left( 1-\tau^2\right)\\
    \tau^3 &\geq -e\,x -1 \geq -\tau^2\\
    \left(W_0(x)+1\right)^3 &\leq e\,x + 1 \leq \left(W_0(x)+1\right)^2\\
\end{split}
\end{equation}
Applying cubic and square roots, respectively, and subtracting 1 leads to the inequalities in \eqref{3:eq:boundlamb}.
\end{proof}
In settings like \eqref{2:eq:lambdatau}, the  Lambert function is expressed in the form:
\begin{equation}
   y - \log\left(1+y\right) = u
    \label{3:eq:lambertlog}
\end{equation}
which is equivalent to \eqref{3:eq:Lambert} for $x = -e^{-1-u}$ and $y = -1 - W(x)$. In that case the bounds \eqref{3:eq:boundlamb} reduce to the form:
\begin{equation}
    \sqrt{1-e^{-u}}-1 \leq W_0(-e^{-1-u}) \leq  \sqrt[3]{1-e^{-u}}-1
    \label{3:eq:boundlamb2}
\end{equation}
We introduce a final inequality relating the values of both branches $W_0(x)$ and $W_{-1}(x)$ for negative values of $x$:
\begin{theorem}
For any $x \in [-e^{-1},0)$, the following inequality holds:
\begin{equation}
    \frac{W_0(x) + 1}{-W_0(x)} \geq \frac{W_{-1}(x) + 1}{W_{-1}(x)}
    \label{3:eq:frac_ineq}
\end{equation}
\label{3:thm:ineq}
\end{theorem}
\begin{proof}
We start  with the change of variables $W_0(x) = -1-\sigma$ and $W_{-1} = -1 -\tau$, with $\sigma \in (-1,0]$ and $\tau \in [0,\infty)$ by construction. By definition of the Lambert function we have that:
\begin{equation}
    \sigma - \log(1+\sigma) = \tau - \log(1+\tau) = -1-\log(-x)
\end{equation}
Inequality \eqref{3:eq:frac_ineq} can also be written with respect to $\sigma$ and $\tau$, leading to:
\begin{equation}
        \frac{-\sigma}{1+\sigma} \geq \frac{\tau}{1+\tau} \quad \iff \quad
     -1  - \frac{1}{\sigma}  \leq  \frac{1}{\tau}+1 \quad \iff \quad
     -\sigma \geq \frac{\tau}{1+ 2\tau} 
    \label{3:eq:ineq_sigtau}
\end{equation}
Now define $\nu\left(\tau\right) = \frac{-\tau}{1+ 2\tau} $ and $f(\tau) = \nu\left(\tau\right) - \tau -\log\left(1+\nu\left(\tau\right) \right) + \log\left(1+\tau\right)$. We have that $f(\tau) \left |_{\tau =0} = 0 \right .$ and its derivative its given by:
\begin{equation}
    \frac{-4\tau^2}{\left(1+2\tau\right)^2} \leq 0
\end{equation}
Being non-increasing and crossing 0 at $\tau=0$, we can conclude that $f(\tau)\leq 0$ for $\tau\geq 0$. This establishes that:
\begin{equation}
    \begin{split}
        \nu\left(\tau\right) -\log\left(1+\nu\left(\tau\right) \right) &\leq \tau - \log\left(1+\tau\right)=  \sigma - \log(1+\sigma)
    \end{split}
\end{equation}
The function $y - \log(1+y)$ is decreasing for negative $y$, which leads to $\sigma \leq \nu\left(\tau\right)$ and:
\begin{equation}
    -\sigma \geq -\nu\left(\tau\right) = \frac{\tau}{1+ 2\tau} 
\end{equation}
Due to this result, the chain of inequalities in \eqref{3:eq:ineq_sigtau} holds, and thus \eqref{3:eq:frac_ineq} also holds, finishing the proof.
\end{proof}

%% file: Tex/031-Section31.tex
\section{Final Bound}
At this point we return to the optimization problem \eqref{2:eq:g_eps}. Combining the structure of its constraints and Theorem \ref{3:thm:ineq} we show that its optimizers are non-positive:

\begin{lemma}
The optimizers $\lambda^*_i$ in problem \eqref{2:eq:g_eps} are non-positive.
\label{31:lemma:lambdbnd}
\end{lemma}
\begin{proof}
We start the proof with problem \eqref{2:eq:g_eps} and introduce the  variables $\tau_i$ and the parameterization of $\lambda_i$ using the Lambert function presented in \eqref{2:eq:lambdatau}:
\begin{equation}
\begin{split}
       g\left(\epsilon\right) = \operatorname*{max}_{\lambda_i,\tau_i}\; \sum_i \left(\frac{\lambda_i}{1+\lambda_i}\right)^2 \quad \text{s.t.}\;&\;\lambda_i = - W(-e^{-\left(1+\tau_i\right)}) -1,\;
      \\
       & \lambda_i > -1,\;\sum_i^N \lambda_i \leq 0, \; \sum_i^N \tau_i \leq \epsilon
\end{split}
    \label{3:eq:g_eps}
\end{equation}
Given $\tau_i$, $\lambda_i$ can take only one of two values, given by the two branches of the Lambert function, i.e. either $\lambda_i = - W_0(-e^{-\left(1+\tau_i\right)}) -1$ or $\lambda_i = - W_{-1}(-e^{-\left(1+\tau_i\right)}) -1$. Due to the range of each branch of the Lambert function, the former implies that $\lambda_i\leq0$ while the latter implies $\lambda_i\geq0$. Assume by contradiction that the optimal $\lambda_i$ is non-negative and thus $\lambda_i = - W_{-1}(-e^{-\left(1+\tau_i\right)}) -1$. Then switching $\lambda_i$ to the principal branch $\lambda_i \to \hat{\lambda}_i = - W_0(-e^{-\left(1+\tau_i\right)}) -1$ would result in an equal or larger objective, since from Theorem \ref{3:thm:ineq} we have that:
\begin{equation}
    \left(\frac{\hat{\lambda}_i}{1+\hat{\lambda}_i}\right)^2=\left(\frac{W_0(x) + 1}{W_0(x)}\right)^2 \geq  \left(\frac{W_{-1}(x) + 1}{W_{-1}(x)}\right)^2 = \left(\frac{\lambda_i}{1+\lambda_i}\right)^2
\end{equation}
We check that $\hat{\lambda}_i$ is also a feasible solution to \eqref{3:eq:g_eps} by analyzing the two constraints affecting $\lambda_i$. Regarding the constraint $\sum_i^N \lambda_i \leq 0$, the fact that $\lambda_i \geq \hat{\lambda}_i$ means that the sum over all $\lambda$ cannot increase by moving $\lambda_i$ to the principal branch and thus the constraint will be preserved. For the second constraint $\lambda_i > -1$, we note that $\tau_i$ is non-negative by construction due to the range of the function $\tau_i= \lambda_i - \log(1+\lambda_i)$, and also bounded above due to the last constraint $\sum_i^N \tau_i \leq \epsilon$. As a result the argument of the Lambert function is strictly negative, i.e. $-e^{-\left(1+\tau_i\right)}<0$, which results in $\hat{\lambda}_i = - W_0(-e^{-\left(1+\tau_i\right)}) -1 > -1$. 
\end{proof}
From the last lemma we can reduce problem \eqref{2:eq:g_eps} to the following form:
\begin{equation}
\begin{split}
       g\left(\epsilon\right) = \operatorname*{max}_{\lambda_i}\; \sum_i \left(\frac{\lambda_i}{1+\lambda_i}\right)^2 \quad \text{s.t.}\;\sum_i^N \lambda_i - \log(1+\lambda_i)\leq \epsilon,\;
      0\geq\lambda_i \;\forall i
\end{split}
    \label{3:eq:g_eps2}
\end{equation}
where we have used the non-positivity of the optimizer to transform the trace constraint into an element-wise one, and the finitiness of $\epsilon$ to remove the lower bound on each $\lambda_i$. From this new formulation, the global optimizer of \eqref{3:eq:g_eps2} can be found in closed form, as shown in the next Lemma:
\begin{lemma}
The optimizer $ \bm{\lambda}^*$ to problem \ref{3:eq:g_eps2} is of the form
\begin{equation}
    \bm{\lambda}^* = \{- W_0(-e^{-\left(1+\epsilon\right)}) -1,0,0\dots,0\}
\end{equation}
\label{3:lemma:optilamb}
\end{lemma}
\begin{proof}
We start the proof analyzing the KKT conditions of \eqref{3:eq:g_eps2}. For brevity, define $f(\bm{\lambda}) = \sum_i \left(\frac{\lambda_i}{1+\lambda_i}\right)^2$ and $h(\bm{\lambda}) = \sum_i^N \lambda_i - \log(1+\lambda_i)-\epsilon$. Denote by $\alpha$ the dual variable associated with the constraint $h(\bm{\lambda})\leq0$ and $\beta_i$ the one associated with each constraint $\lambda_i\leq0$. Then the stationarity and complementary slackness KKT conditions for \eqref{3:eq:g_eps2} can be expressed as:
\begin{equation}
    \begin{split}
        -\frac{\delta f}{\delta \lambda_i^*}+\alpha\;\frac{\delta h}{\delta \lambda_i^*}+\beta_i\; = 0, \quad \lambda_i^*\;\beta_i=0,\quad \text{for all }i,\quad \alpha\;h(\bm{\lambda}^*)=0
    \end{split}
\end{equation}
Due to complementary slackness, all strictly negative $\lambda_i$ will satisfy $\beta_i=0$ and their stationarity condition simplifies to:
\begin{equation}
    \begin{split}
        -\frac{\delta f}{\delta \lambda_i^*}+\alpha\;\frac{\delta h}{\delta \lambda_i^*}+\beta_i\; = \frac{-\lambda_i^*}{\left(1+\lambda_i^*\right)^3}+\alpha\;\frac{\lambda_i^*}{1+\lambda_i^*}=0\implies \left(1+\lambda_i^*\right)^2\alpha =1
    \end{split}
\end{equation}
This result has two implications: first it implies that $\alpha>0$, which by complementary slackness implies that $h(\bm{\lambda}^*)=0$. Second, as the dual variable $\alpha$ does not depend on $i$, all strictly negative $\lambda_i^*$ must share the same value. Assume there the first $n$ eigenvalues $\lambda_i^*$ are strictly negative and share the value $\lambda_i^*=\lambda^*$, and the remaining $N-n$ are $0$. Then the objective $f(\bm{\lambda})$ and constraint $h(\bm{\lambda}^*)=0$ simplify to:
\begin{equation}
    \begin{split}
       f(\bm{\lambda})=f(\lambda^*) = n \left(\frac{\lambda^*}{1+\lambda^*}\right)^2, \quad h(\bm{\lambda})=h(\lambda^*) = n\left(\lambda^* - \log(1+\lambda^*)\right)- \epsilon=0
    \end{split}
\end{equation}
Finding the value of $\lambda^*$ is then equivalent to finding the value of $n$.  Solving for $n$ in $h(\bm{\lambda})$ and using this value  in $f(\bm{\lambda})$ leads to:
\begin{equation}
    \begin{split}
       f(\lambda^*) = \frac{\epsilon}{\lambda^* - \log(1+\lambda^*)} \left(\frac{\lambda^*}{1+\lambda^*}\right)^2
    \end{split}
\end{equation}
The value of $\lambda^*$ will be the one that maximizes the above expression while satisfying $h(\lambda^*)=0$ for an integer value of $n$ between $1$ and $N$ (we ignore the possibility of $n=0$ as it only applies to the trivial case $\epsilon=0$). To find the value of $n$, we analyze the derivative of $f(\lambda^*)$ on the interval $-1<\lambda^*<0$:
\begin{equation}
    f^\prime(\lambda^*) = \epsilon\frac{\lambda^*\left[2(\lambda^* - \log(1+\lambda^*))-(\lambda^*)^2 \right]}{\left(1+\lambda_i^*\right)^3(\lambda^* - \log(1+\lambda^*))^2}
\end{equation}
The denominator is strictly positive on the interval and the numerator has the same sign as $\lambda^*$, since, as shown by the Taylor series of $\lambda^* - \log(1+\lambda^*)$:
\begin{equation}
    \lambda - \log(1+\lambda) = \sum_{k=2}^\infty (-1)^k\frac{\lambda^k}{k} > \frac{\lambda^2}{2}\quad\text{for}\quad-1<\lambda<0
\end{equation}
Thus  $f(\lambda^*)$ is a decreasing function on the interval $-1<\lambda^*<0$. In order to maximize $f(\lambda^*)$, the value of $\lambda^*$ will be the smallest possible while still satisfying the constraint $h(\lambda^*)=0$ for an integer value of $n$ between $1$ and $N$. Expressing this constraint as:
\begin{equation}
    \lambda^* - \log(1+\lambda^*)= \frac{\epsilon}{n}
\end{equation}
and noting that $\lambda^* - \log(1+\lambda^*)$ is also a decreasing function on the interval of interest, the most negative $\lambda^*$ corresponds to the smallest $n$, i.e. $n=1$. Finally, isolating the value of $\lambda^*$ through the Lambert function:
\begin{equation}
    \lambda_1^* = \lambda^* = - W_0(-e^{-\left(1+\epsilon\right)})-1,\quad \lambda_i^* =0\quad \text{for }i>1
\end{equation}
yields the optimal solution and finishes the proof.
\end{proof}

Using the optimizer for problem \eqref{3:eq:g_eps2} obtained in Lemma \ref{3:lemma:optilamb} we can derive the bound presented in \eqref{1:eq:finalbound}:
\begin{theorem}
Let $X_*$  be the optimizer of \eqref{1:eq:minlogdet} and  $X_f$  a feasible solution for \eqref{1:eq:minlogdet}, respectively. If the optimality gap of these solutions satisfies: 
\begin{equation}
    \epsilon \geq \log\left(|X_*|\right) -\log\left(|X_f|\right) \geq 0
\end{equation}
then the Frobenius norm of their difference is bounded as:
\begin{equation}
     ||X_*-X_f||_F^2 \leq ||X_f||_2^2 \left(\frac{\sqrt[3]{1-e^{-\epsilon}}}{1-\sqrt[3]{1-e^{-\epsilon}}}\right)^2
\end{equation}
\label{3:thm:finalbound}
\end{theorem}
\begin{proof}
The proof follows directly from combining the inequalities in \eqref{2:eq:norm_ineq} and the solution of the optimization problem \eqref{2:eq:g_eps}, leading to:
\begin{equation}
\begin{split}
        ||X_f-X_*||_F^2 \leq ||X_f||_2^2|| \left(Q+I\right)^{-1}Q||_F^2\,
        &\leq ||X_f||_2^2\left( \frac{\lambda^*}{1+\lambda^*} \right)^2\\
        &= ||X_f||_2^2\left( \frac{W_0(-e^{-1-\epsilon}) + 1}{-W_0(-e^{-1-\epsilon})} \right)^2
\end{split}
\end{equation}
Finally, applying the bounds developed in Theorem \ref{3:thm:bounds} and \eqref{3:eq:boundlamb2} leads to:
\begin{equation}
        ||X_f-X_*||_F^2 \leq  ||X_f||_2^2 \left(\frac{\sqrt[3]{1-e^{-\epsilon}}}{1-\sqrt[3]{1-e^{-\epsilon}}}\right)^2
\end{equation}
\end{proof}

%% file: Tex/04-.Section4.tex
\section{Numerical Application}
As an example application, we apply the bounds derived in Theorem \ref{3:thm:bounds} in the context of interior point method solvers. More explicitly, we will study the use of termination criteria based on the normalized Euclidean distance between the optimizer $X_*$ of \eqref{1:eq:maxdet} and a suboptimal solution $X_f$, instead of the standard criteria based on the optimality gap $\epsilon$ of the objective function. As discussed in the Introduction, in many applications the parameters of interest are entries of $X$ or  functions of those entries. Hence, these novel termination criteria can be directly related to the entry-wise discrepancies between $X_*$ and $X_f$ and thus provide a way to terminate the algorithm based on the suboptimality of these parameters.\\
\\
From Theorem \ref{3:thm:finalbound} we have that the normalized Euclidean distance is bounded as:
\begin{equation}
    \frac{||X_*-X_f||_F}{||X_f||_2} \leq  \left|\frac{\sqrt[3]{1-e^{-\epsilon}}}{1-\sqrt[3]{1-e^{-\epsilon}}}\right| = g(\epsilon)
    \label{4:eq:termibound}
\end{equation}
Due to the self-concordance of the log-determinant, interior point methods can provide bounds on $\epsilon$ for the suboptimal solutions $X_f$, but without knowledge of $X_*$ (which by definition is never available) cannot generally provide bounds on the Euclidean distance between $X_f$ and $X_*$. Bounds of the form \eqref{4:eq:termibound} are then of interest, as they allow to connect $||X_*-X_f||_F$ to the optimality bound $\epsilon$  provide by the interior point algorithm.\\
\\
As a test bed, we study this termination criterion on the minimum volume ellipsoid covering problem, an instance of the $\textit{maxdet}$ problem. This problem is defined as follows: given a set $\mathcal{Y} = \{y_1,\dots,y_M\}$ of $M$ vectors $y_i\in\mathbb{R}^n$ denoting points in an $N$-dimensional space, the minimum volume ellipsoid covering problem tries to find the minimum volume ellipsoid that contains all those points \cite{Vander1998}. Describing the ellipsoid as $\varepsilon = \{y\, | \,||X y+b||_2\leq 1\}$ for $X\in \mathcal{S}_+^N$, the volume of the ellipsoid $\varepsilon$ is proportional to $|X^{-1}|$. Minimizing $|X^{-1}|$ (or equivalently its logarithm) is equivalent to the maximization of $|X|$, from which follows that the minimum volume ellipsoid covering problem is an instance of the more general $\textit{maxdet}$ problem \eqref{1:eq:maxdet} and is defined as:
\begin{equation}
    \begin{split}
        X_* = \operatorname*{argmin}_{X\succeq 0, \,b}\; -\log\left(|X|\right)\quad\text{s.t.}\quad||X\, y_i -b||_2^2 \leq 1\quad \forall \;y_i \in \mathcal{Y} 
    \end{split}
    \label{4:eq:minellip}
\end{equation}
To test the quality of the bound \eqref{4:eq:termibound}, we generate a set of points $\mathcal{Y}$ and numerically solve \eqref{4:eq:minellip} using the Matlab package CVX \cite{cvx,gb08} and the solver SDPT3 \cite{toh1999sdpt3}. We first solve \eqref{4:eq:minellip} using the best tolerance available in CVX, i.e. $\delta \approx 1.5 \times 10^{-8}$, and take the result of that numerical optimization to be the global optimum $X_*$. Afterwards, we solve \eqref{4:eq:minellip} using a set of higher tolerances ranging from $\delta = 1$ to $\delta = 10^{-8}$, leading to suboptimal solutions which we label as $X_f$. For each suboptimal optimizer $X_f$, we compute the normalized Euclidean distance as in \eqref{4:eq:termibound} and the upper bound $g(\epsilon)$ from the optimality bound $\epsilon = \log\left(|X_*|\right)-\log\left(|X_f|\right)$. The set $\mathcal{Y}$ is generated by sampling $M = 100$ points from a $N$-dimensional standard normal distribution with $N=50$.\\
\\
\begin{figure}
    \centering
    \includegraphics[width = 8cm]{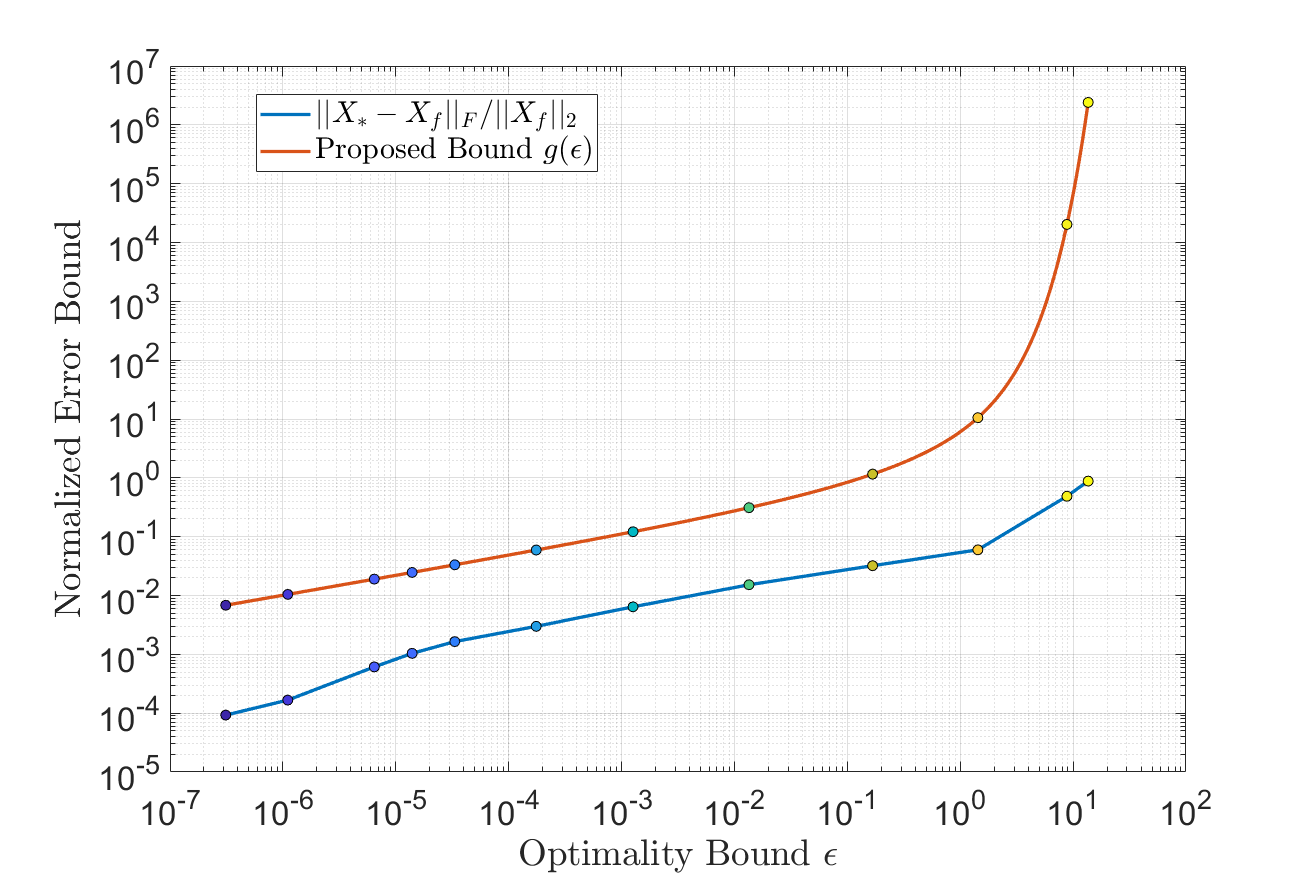}
    \caption{Normalized error bound $ ||X_*-X_f||_F /||X_f||_2$ of suboptimal solutions $X_f$ computed at different tolerances $1\geq \delta \geq 10^{-8}$ versus their optimality bound $\epsilon = \log\left(|X_*|\right)-\log\left(|X_f|\right)$, in blue, and the proposed upper bound \eqref{4:eq:termibound} in red. The proposed bound grows unbounded for large $\epsilon$, but converges to $g(\epsilon) \approx \sqrt[3]{\epsilon}$ for small $\epsilon$, providing a simple estimate of the normalized error from the optimality bound $\epsilon$.}
    \label{fig:fig4}
\end{figure}
Figure \ref{fig:fig4} shows, in blue, the normalized Euclidean distance of each suboptimal solution $X_f$ as a function of their optimality bound $\epsilon$, as well as the value provided by the upper bound $g(\epsilon)$, in red. While the bound seems to diverge for large values of $\epsilon$, it converges to the the limit behavior $g(\epsilon) \approx \sqrt[3]{\epsilon}$ as $\epsilon \to 0$. Despite being a conservative bound, the bound can be used to terminate the optimization procedure (or, conversely, to add further steps to refine the solution) by comparing it to a desired Frobenius norm value, providing a guaranteed normalized error between $X_f$ and $X_*$ at the end of the optimization algorithm.

%% file: Tex/05-.Section5.tex
\section{Future Work and Conclusions}\label{4:sec}

In this work we have developed tractable bounds on the Frobenius norm of the error between feasible suboptimal solutions to the \textit{maxdet} problem \eqref{1:eq:minlogdet} and its global optimizer. By combining the traditional bounds on the log-determinant \eqref{1:eq:costbound} provided by duality theory and the interpretabiltiy of the Frobenius norm, these bounds provide a way to study the quality of approximations to the analytic center of an LMI set with respect to the elements of the LMI argument, which are usually the parameters of interest for many applications.\\
\\
The derivation of these bounds has been done through analysis of the KKT conditions of the \textit{maxdet} problem as well as through the development of novel bounds on the Lambert function $W_{0}(x)$. In future work we will explore the extension of the bounds derived in Section \ref{3:sec} to expressions dependant on the norm of the optimizer $||X_*||_2^2$ as well as metrics beyond the Frobenius norm, like the geometry-aware Jensen-Bregman LogDet divergence.